\UseRawInputEncoding

\newenvironment{dedication}
{
\itshape
\raggedleft
}
{\par
\vspace{3mm}
}

\documentclass{amsart}
\usepackage{amsmath,amssymb}

\newtheorem{theorem}{Theorem}

\usepackage{epsfig}
\usepackage{url}
\usepackage{amssymb}
\usepackage{amsmath}
\usepackage{amsfonts}

\def\Dbar{\leavevmode\lower.6ex\hbox to 0pt{\hskip-.23ex \accent"16\hss}D}

\def\bZ{{\mbox{\bf Z}}}

\begin{document}

\title{Orthogonal designs and symmetric Hadamard matrices}

\author {Dragomir {\v{Z}}. {\Dbar}okovi{\'c}}
\address{University of Waterloo,
Department of Pure Mathematics,
Waterloo, Ontario, N2L 3G1, Canada}
\email{dragomir@rogers.com}
\date{}

\begin{abstract}
For a prime power $q\equiv 3 \pmod 8$ we construct skew-type and symmetric
orthogonal designs with parameters $(1+q:~1,q)$.
Our main result is the construction of symmetric Hadamard matrices of order
$q(1+q)$ for the same values of $q$.
\end{abstract}

\maketitle

\begin{dedication}
In memory of Nikolay Alekseevich Balonin, my friend and collaborator.
\end{dedication}

\section{Introduction}
For a square matrix we say that it is of {\em skew type} if all diagonal
entries are equal to each other and if we replace all diagonal entries
by 0 the matrix becomes skew-symmetric.
Let $q\equiv 3 \pmod 8$ be a prime power. We construct a symmetric
and a skew-type orthogonal designs of order $1+q$.

By using the symmetric orthogonal design of order $1+q$, we construct 
symmetric Hadamard matrix of order $q(1+q)$. This gives a new infinite series
of symmetric Hadamard matrices. In particular, the cases $q=27$ and $q=43$
give the first examples of symmetric Hadamard matrices of orders $4\cdot 189$
and $4\cdot 473$. Let us mention that if $q\equiv 3 \pmod 4$ is a
prime power then there exists a Hadamard matrix of order $q(1+q)$
(see \cite[Corollary 1]{Djokovic:LAMA:2017}).

In the last section we present some new results on Legendre pairs of length
$111$.
\section{Two arrays}
The Goethals-Seidel array (GS-array) was originally invented
in 1970 by J. M. Goethals and J. J. Seidel \cite{GS:JAMS:1970} in order to
construct a first example of a skew-Hadamard matrix of order 36. 
Since that time, it has been used extensively by many researchers to
construct several infinite series of Hadamard matrices of
various kinds. However the GS-array is not suitable for the construction
of symmetric Hadamard matrices. For the construction of 
this type of matrices there is a different array discovered by N. A. Balonin
in 2014, who gave it the name {\em propus array}.
In 2017 it was published in the joint paper of J. Seberry and N. A. Balonin
\cite{SB:2017} and used to construct two infinite series of symmetric Hadamard
matrices. In my opinion that array should be called 
{\em Balonin array}, and I will use this new name in this article. 
For reader's convenience we display these two arrays:

$$ \left[ \begin{array}{cccc}
       A &  BR &  CR &  DR \\
     -BR &  A  & -RD &  RC \\
     -CR &  RD &  A  & -RB \\
     -DR & -RC &  RB &  A  \\ \end{array} \right] {\rm Goethals-Seidel~ array,} $$

$$ \left[ \begin{array}{cccc}
      -A &  BR &  CR &  DR \\
      BR &  RD &  A  & -RC \\
      CR &  A  & -RD &  RB \\
      DR & -RC &  RB &  A  \\ \end{array} \right]  {\rm Balonin~ array.} $$

In both arrays the matrices $A,B,C,D$ of order $v$ are circulants and 
$R$ is the back-circulant identity matrix of the same order. 
The matrices $A,B,C,D$ are usually constructed from difference families,
with suitable parameters, in cyclic groups. However, these arrays can
be adapted to construct Hadamard matrices by using suitable 
difference families over an arbitrary finite abelian group $G$ of order $v$.
In that case the matrices $A,B,C,D$ should be $G$-invariant.
This means that if, for instance, $A=[a_{x,y}]$ with $x,y\in G$ then
$a_{x+z,y+z}=a_{x,y}$ for all $x,y,z\in G$.
The matrix $R$ also must be modified accordingly. 
In all cases $R$ is an involutive and symmetric permutation matrix. Moreover,
if a matrix $A$ is G-invariant then $RAR=A^T$ where $T$ denotes transposition.
For more details see \cite[Section 9]{DK:Computational methods:2019}. 
(There is a misprint there in formula (24): the letter A should be
replaced by G.) 

\section{Two infinite series of orthogonal designs}

We denote by $I_n$ the identity matrix of order $n$.
Let $x_1,x_2,\ldots,x_s$ be commuting independent variables.
Let us consider square matrices $X$ of order $n$ whose entries
belong to $\{0,\pm x_1,\pm x_2,\ldots,\pm x_s\}$. We say that 
such a matrix is an {\em orthogonal design} (OD) if
$$
XX^T=(\sum_{i=1}^s k_i x_i^2)I_n
$$
for some integers $k_i > 0$. In that case we say that the orthogonal design
$X$ has parameters $(n;~ k_1,k_2,\ldots,k_s)$.  
For the theory of orthogonal designs we refer the interested reader to 
\cite{CK:2007} and \cite{JS:2017}.

\begin{theorem}
If $q\equiv 3 \pmod{8}$ is a prime power then there exist
an OD $X$ of skew type with parameters $(1+q:~ 1,q)$ and a symmetric OD $Y$
with the same parameters.
\end{theorem}

\begin{proof}
Since $q\equiv 3 \pmod{8}$ the number $n=(1+q)/4$ is an odd integer.
The main result of the paper \cite{XXSW:2005} can be rephrased as follows:
In the cyclic group $\bZ_n=\bZ/n\bZ$ of order $n$ there 
exists a difference family consisting of four blocks $X_i$, 
$i=0,1,2,3$, with parameters 
$(n;~ k_0=(n-1)/2,k_1,k_2,k_3;~ \lambda)$ such that:

\noindent
(a) $|X_i|=k_i$ for all $i$; \\
(b) $k_0+k_1+k_2+k_3=n+\lambda$; \\
(c) $X_1=X_2$ (and so $k_1=k_2$); \\
(d) the block $X_0$ is skew i.e., $0\notin X_0$ and for each
$i\in X_0$ we have $n-i\notin X_0$; \\
(e) the block $X_3$ is symmetric i.e., $-X_3=X_3$.

Let $x$ and $y$ be two commuting variables. 
For $i=0,1,2,3$ let $A_i$ be the cyclic matrix of order $n$
with first row
$[\alpha_{i,0},\alpha_{i,1},\ldots,\alpha_{i,n-1}]$ where
$\alpha_{i,j}=-1$ if $j\in X_i$ and $\alpha_{i,j}=1$ otherwise.
As $X_1=X_2$ we have $A_1=A_2$.

Let $X$ be the matrix of order $4n$ obtained by plugging the matrices 
         $$(x-y)I_n + yA_0,~ yA_1,~ yA_2,~ yA_3$$
(in that order) into the Goethals-Seidel array. Then we have
$$ XX^T = (x^2+qy^2)I_{4n},$$ 
i.e., $X$ is an $OD(1+q;~1,q)$. Moreover, $X$ is also of skew type.

Let $Y$ be the matrix of order $4n$ obtained by plugging the matrices 
         $$yA_3,~yA_1,~yA_2,~(x-y)I_n+yA_0$$ 
(in that order) into the Balonin array. Then we have 
$$ YY^T = (x^2+qy^2)I_{4n}.$$
Since $A_3$ is a symmetric matrix, it follows that $Y$ is
a symmetric $OD(1+q;~1,q).$
\end{proof}

\noindent
{\bf Example.} As in \cite{XXSW:2005} we take $q=27$ and so $n=7$. 
We shall use the skew difference set $\{1,2,4\}$ 
as each of the blocks $X_0, X_1, X_2$. Further, we set 
$X_3=\{0\}$ and we obtain a difference family with parameters 
$(7;~3,3,3,1; ~3)$.
This difference family satisfies all five conditions (a-e) listed above. 
The first three matrices $A_i$ have the same first row, namely 
$[1,-1,-1,1,-1,1,1]$. The first row of $A_3$ is $[-1,1,1,1,1,1,1]$. 
It is now straightforward to compute the desired OD's $X$ and $Y$ 
having the same parameter set, namely $(28;~ 1,27)$. 
One just needs to plug the circulants 
$$(x-y)I_n + yA_0, yA_1, yA_2, yA_3$$
into the Goethals-Seidel and the Balonin arrays respectively 
(by using the orderings mentioned above). Then we get the desired OD's:
the skew-type OD $X$

\begin{equation}
\left[ \begin{array}{cccc}
       (x-y)I_7+yA_0 & yA_1R         & yA_2R         & yA_3R \\
         -yA_1R      & (x-y)I_7+yA_0 & -yRA_3        & yRA_2 \\
         -yA_2R      & yRA_3         & (x-y)I_7+yA_0 & -yRA_1 \\
         -yA_3R      & -yRA_2        & yRA_1         & (x-y)I_7+yA_0 \\
      \end{array}r
\right] 
\end{equation}

and the symmetric OD $Y$
\begin{equation} 
\left[ \begin{array}{cccc}
       -yA_3 & yA_1R        & yA_2R         & (x-y)R+yA_0R \\
       yA_1R & (x-y)R+yRA_0 & yA_3          & -yRA_2 \\
       yA_2R & yA_3         & -(x-y)R-yRA_0 & yRA_1 \\
(x-y)R+yA_0R & -yRA_2       & yRA_1         & yA_3 \\
      \end{array}
\right].
\end{equation}

\section{A new infinite series of symmetric Hadamard matrices}

\begin{theorem}
If $q\equiv 3 \pmod{8}$ is a prime power then there exists 
a symmetric Hadamard matrix of order $q(1+q)$.
\end{theorem}

\begin{proof}
By Theorem 1 there exists a symmetric OD $Y$ with parameters 
$(1+q:~ 1,q)$ and variables $x$ and $y$. This matrix $Y$ has order $1+q$.
In each row and column of $Y$, the variable $x$ occurs only once while 
$y$ occurs $q$ times. We shall now blow up $Y$ to obtain a symmetric
$\{+1,-1\}$-matrix $H$ of order $q(1+q)$. This is accomplished simply by
replacing each occurrence of $x$ and $y$ in $Y$ by the commuting symmetric 
$\{+1,-1\}$-matrices $J_q$ and $D_q$ of order $q$. $J_q$ is all one matrix,
i.e. all entries of $J_q$ are equal 1. We set
\begin{equation}
D_q=(I_q+Q)R
\end{equation}
where $Q$ is known as the {\em Paley core} matrix and
$R$ is the direct sum of the matrix $[1]$ of order 1
and the back-diagonal identity matrix $S$ of order $q-1$.

To construct the matrix $Q$, we choose
a generator $g$ of the multiplicative group of the finite field
$F_q$ of order $q$. We list the elements of $F_q$ in the following order: 
$$ 0, g^0, g^1, \ldots, g^{q-3}, g^{q-2}. $$
We use this list as labels for the rows and columns of $Q$.
For $r,s\in F_q$ the $(r,s)$ entry of $Q$ is equal to $\chi(r-s)$ where
$\chi$ is the quadratic character of $F_q$. We recall that $\chi(0)=0$ 
and $\chi(g^i)=(-1)^i$.
Obviously $J_q$ is symmetric, and $Q$ is skew-symmetric and $G$-invariant where $G$ is the additive group of $F_q$. 
Since $RQR=Q^T=-Q$, one can easily verify that $D_q$ is also symmetric.
Further, we have $J_q D_q=J_q=D_q J_q$ and so $J_q$ and $D_q$ commute. 
It follows that $H$ is a symmetric $\{+1,-1\}$-matrix.

Let us partition $H$ into blocks of size $q$, and do the same for $H^T$ and
the product $HH^T$. Since $Y$ is an OD, it is obvious that all off-diagonal
blocks of $HH^T$ are zero and all diagonal blocks are equal.

To finish the proof, it remains to show that $H$ is a Hadamard matrix. 
Since $Y$ is an OD and $J_q$ and $D_q$ commute all off-diagonal blocks
of $HH^T$ vanish. Note that all the diagonal blocks of $HH^T$ are equal.
The inner product of any two rows of $J_q$ is equal $q$. On the other hand, 
the inner product of any two rows of $D_q$ is equal to $-1$.
This follows from the fact that $QQ^T=qI_q-J_q$ (see \cite[Chapter 18]
{LintWilson:1992}). Indeed, since $R^2=I_q$ we have

$$ D_q D_q^T=(I_q+Q)(I_q-Q)=(q+1)I_q-J_q.$$

Hence the inner product of any two rows of $H$ selected from the first
block-row of $HH^T$ is equal 0 because $1 \cdot q+q \cdot (-1)=0$.
(Recall that in each row of $Y$ the variable $x$ occurs only once and $y$
occurs $q$ times.) Hence the first block of $HH^T$ is a diagonal matrix.
Clearly this block is equal to $q(1+q)I_q$. Since all diagonal blocks of $HH^T$
are equal to each other, the proof is completed.
\end{proof}

By using Maple (Maplesoft, a division of Waterloo Maple Inc., 2025, Waterloo, 
Ontario) we have constructed explicitly symmetric Hadamard matrices
of order $q(1+q)$ for $q=3,11,19,27,43,59,67,107,131$.

\section{Addendum on Legendre pairs of length $111$}

Legendre pairs (LegP) are difference families in finite abelian groups of
odd order $v$ consisting of two blocks, both of size $(v-1)/2$. Hence its
parameters are $(v; (v-1)/2,(v-1)2; (v-3)/2)$. If the group is cyclic, we
say that its LegP are {\em cyclic}. In the case of cyclic LegP it is customary
to refer to $v$ as the {\em length} of the LegP. It is conjectured that
cyclic LegP exist for all odd $v$.
There are four known infinite series of cyclic LegP (see \cite{FGS:2001}). 
One of them uses so called Szekeres difference sets. For $v=111$ the Szekeres
difference sets $M$ and $N$ exist and form the following LegP:
$$\displaylines{
M=\{[6,9,11,12,13,16,22,23,30,31,32,34,35,38,39,40,43,45,50,52,54,56,58, \hfill\cr
60,62,63,64,65,67,69,70,74,75,78,82,83,84,85,86,87,90,91,92,93,94,96,97, \hfill\cr
101,103,104,106,107,108,109,110\}, \hfill \cr
N=\{0,1,2,4,5,6,7,10,11,14,16,18,21,24,30,31,32,33,35,36,39,43,44,45,49, \hfill\cr
51,52,54,57,59,60,62,66,67,68,72,75,76,78,79,80,81,87,90,93,95,97,100, \hfill\cr
101,104,105,106,107,109,110\}. \hfill \cr}
$$
Note that $M$ is skew while $N$ is symmetric. Until recently this was the unique 
(up to equivalence) known example of LegP of length $111$.
Nick (short for Nikolay) sent me (see \cite{Nick:2025}) the following LegP that 
he found:
$$\displaylines{
\{0,4,5,6,8,9,12,14,21,22,23,25,26,28,29,31,32,33,38,40,41,45,47,50,51, \hfill\cr
53,54,55,56,58,60,61,66,67,68,72,77,80,81,83,86,88,90,91,92,96,98,99,102, \hfill\cr
103,104,105,106,108,109\}, \hfill \cr
\{1,2,5,8,9,10,11,12,15,16,17,20,22,23,25,27,28,31,32,35,36,37,39,41,46, \hfill\cr
48,49,50,51,52,54,56,57,58,59,66,72,76,77,80,88,89,90,91,92,94,96,98,100,\hfill\cr
101,103,104,105,109,110\}. \hfill \cr}
$$
I verified that these two LegP are not equivalent. Subsequently I used the subgroup
$H=\{1,10,100\}$ to search for LegP of length $111$ where each of the blocks is
$H$-invariant. I found 13 nonequivalent LegP. To save space,
we shall use the list of $H$-orbits. Apart from the trivial orbit $\{0\}$
there are two more singleton orbits, namely $\{37\}$ and $\{74\}$, 
and 36 orbits of size 3. We use the 
integers $0,1,2,\ldots,37$ to label the 38 nontrivial orbits. Note that the
$H$-orbit with label 0 is the subgroup $H$ itself.

\begin{table}[!h]
\centering
\caption{Labelling of $H$-orbits}
\begin{tabular}{|r|l|r|l|r|l|}
	\hline
 0) & 1,10,100 & 1) & 11,101,110 & 2) & 2,20,89 \\
 3) & 22,91,109 & 4) & 3,30,78 & 5) & 33,81,108 \\ 
 6) & 4,40,67 & 7) & 44,71,107 & 8) & 5,50,56 \\
 9) & 55,61,106 & 10) & 6,45,60 & 11) & 51,66,105 \\
12) & 7,34,70 & 13) & 41,77,104 & 14) & 8,23,80 \\
15) & 31,88,103 & 16) & 9,12,90 & 17) & 21,99,102 \\
18) & 13,19,79 & 19) & 32,92,98 & 20) & 14,29,68 \\
21) & 43,82,97 & 22) & 15,39,57 & 23) & 54,72,96 \\
24) & 16,46,49 & 25) & 62,65,95 & 26) & 17,35,59 \\
27) & 52,76,94 & 28) & 18,24,69 & 29) & 42,87,93 \\
30) & 25,28,58 & 31) & 53,83,86 & 32) & 26,38,47 \\
33) & 64,73,85 & 34) & 27,36,48 & 35) & 63,75,84 \\
36) & 37 & 37) & 74 &&\\
	\hline
\end{tabular}
\end{table}

Now we can list our 13 LegP by using the above labels. First each label in the list
below should be replaced by the corresponding orbit. Second, each block in
all 13 solutions contains the trivial orbit $\{0\}$ and must be added
separately.

$$\displaylines{
 1)\hspace{.1in} [[0,1,2,3,4,8,9,10,11,14,16,19,20,24,29,32,33,35],  \hfill\cr
\hspace{.5in}     [0,2,6,7,9,11,13,16,18,19,22,23,25,27,28,29,30,33]],   \hfill\cr
 2)\hspace{.1in}  [[1,2,3,9,10,12,13,16,17,22,25,26,27,28,31,32,33,35], \hfill\cr,
\hspace{.5in}     [0,1,3,4,5,6,8,13,14,17,18,20,21,27,28,33,34,35]],    \hfill\cr
 3)\hspace{.1in}  [[0,1,2,6,7,9,10,11,14,15,16,19,22,29,31,32,33,35],  \hfill\cr
\hspace{.5in}     [2,8,9,10,11,12,13,15,16,19,21,23,24,26,27,28,29,33]], \hfill\cr
 4)\hspace{.1in}  [[2,3,10,11,12,16,19,20,21,23,26,27,28,29,31,32,33,35], \hfill\cr
\hspace{.5in}     [0,1,3,4,8,9,11,13,14,18,19,21,22,24,25,27,28,29]],     \hfill\cr
 5)\hspace{.1in}  [[0,3,4,5,6,7,9,10,11,14,16,17,25,26,31,32,33,35],   \hfill\cr
\hspace{.5in}     [3,4,6,8,14,15,19,20,21,22,23,24,25,26,27,29,31,35]],   \hfill\cr
 6)\hspace{.1in}  [[4,7,10,13,14,15,16,17,19,24,27,28,29,30,31,32,33,35],  \hfill\cr
\hspace{.5in}     [1,3,4,8,9,11,13,14,15,16,18,20,21,26,27,29,31,35]],    \hfill\cr
 7)\hspace{.1in}  [[0,2,3,7,8,9,12,14,15,16,22,24,26,27,28,29,32,35],  \hfill\cr
\hspace{.5in}     [4,5,7,8,12,14,15,17,19,23,26,27,28,29,30,31,33,35]],   \hfill\cr
 8)\hspace{.1in}  [[0,1,2,5,6,7,10,12,13,14,17,19,22,24,25,29,30,33],  \hfill\cr
\hspace{.5in}     [0,2,4,8,9,10,11,12,13,14,15,16,17,21,22,26,29,31]],    \hfill\cr
 9)\hspace{.1in}  [[1,2,3,4,5,8,12,13,15,20,21,22,23,24,27,29,33,35],  \hfill\cr
\hspace{.5in}     [0,4,5,8,9,13,15,17,18,19,21,23,24,26,27,28,29,31]],    \hfill\cr
10)\hspace{.1in}  [[1,5,6,8,9,13,14,15,16,17,22,24,25,27,30,31,32,35], \hfill\cr
\hspace{.5in}     [0,1,4,5,11,12,13,16,17,18,19,20,25,27,28,32,33,35]],   \hfill\cr
11)\hspace{.1in}  [[0,4,5,7,9,10,12,13,14,18,19,22,24,25,30,31,33,35], \hfill\cr
\hspace{.5in}     [0,3,6,8,10,11,13,14,15,16,17,19,23,25,26,27,28,29]],    \hfill\cr
12)\hspace{.1in}  [[0,2,4,6,7,10,11,14,18,19,21,23,28,30,31,33,34,35], \hfill\cr
\hspace{.5in}     [3,5,6,7,8,12,13,14,21,23,26,29,30,31,32,33,34,35]],     \hfill\cr
13)\hspace{.1in}  [[1,3,4,7,8,9,10,11,12,14,15,20,21,24,27,29,34,35],  \hfill\cr
\hspace{.5in}      [0,1,2,3,8,11,13,15,17,18,19,24,26,28,29,31,34,35]].  \hfill\cr}
$$

My LegP number 13 is equivalent the one found by Nick.
Since my search covered only a small portion of the whole search space,
it appears that the number of equivalence classes of LegP of length $111$ is huge.
For exhaustive searches of cyclic LegP of lengths $\le 47$ see 
\cite[Table 3]{FGS:2001}.

\section{Acknowledgements}

This research was enabled in part by support provided by Sharcnet (https://shar
cnet.ca) and Compute Canada (www.computecanada.ca).


\begin{thebibliography}{999}

\bibitem{Nick:2025} N. A. Balonin, Private communications May 30 and June 5,
2025.

\bibitem{CK:2007}
R. Craigen and H. Kharaghani, 
Orthogonal designs, in Handbook of Combinatorial Designs, 2nd ed. 
C. J. Colbourn, J. H. Dinitz (eds)  pp. 280-295.
Discrete Mathematics and its Applications (Boca Raton).
Chapman \& Hall/CRC, Boca Raton, FL, 2007.

\bibitem{Djokovic:LAMA:2017}
D. {\v{Z}}. {\Dbar}okovi{\'c},
Generalization of Scarpis' theorem on Hadamard matrices.
Linear and Multilinear Algebra 65 (10) (2017), 1985-1987.

\bibitem{DK:Computational methods:2019}
D. {\v{Z}}. {\Dbar}okovi{\'c}, I. S. Kotsireas,
Computational methods for difference families in finite abelian groups.
Special Matrices 2019, 7: 127-141.

\bibitem{FGS:2001}
R. J. Fletcher, M. Gysin, J. Seberry,
Application of the discrete Fourier transform to the search for
generalised Legendre pairs and Hadamard matrices.
Australas. J. Combin.  23  (2001), 75--86.

\bibitem{GS:JAMS:1970}
J. M. Goethals and J. J. Seidel, A skew-Hadamard matrix of order 36,
J. Austral. Math. Soc. A 11 (1970), 343-344.

\bibitem{LintWilson:1992}
J. H. van Lint and R. M. Wilson, A Course in Combinatorics, Cambridge 
University Press, 1992.

\bibitem{JS:2017}
J. Seberry, Orthogonal Designs,: Hadamard Matrices, Quadratic Forms and
Algebras, Springer Nature, Chaum, 2017.

\bibitem{SB:2017}
J. Seberry and N. A. Balonin, Two infinite families of symmetric Hadamard
matrices, Australas. J. Combin.  69(3) (2017), 349-357.

\bibitem{XXSW:2005}
M. Xia, T. Xia, J. Seberry and J. Wu, An infinite series of Goethals-Seidel
arrays, Discrete Applied Mathematics 145 (2005), 498-504.
\end{thebibliography}
\end {document}